\documentclass{amsart}

\usepackage{amssymb,amsmath,color}
\usepackage{amsthm}
\usepackage{url}
\usepackage{tikz}
\usepackage[enableskew]{youngtab}
\usepackage{ytableau,varwidth
}
\allowdisplaybreaks

\def\({\left(}
\def\){\right)}

\newtheorem{theorem}{Theorem}[section]
\newtheorem{corollary}[theorem]{Corollary}
\newtheorem{proposition}{Proposition}
\newtheorem{conjecture}{Conjecture}

\newtheorem{lemma}[theorem]{Lemma}

\theoremstyle{definition}

\newtheorem{remark}{Remark}

\catcode`\@=11

\thinlines
\newskip\Einheit \Einheit=0.6cm
\newcount\xcoord \newcount\ycoord
\newdimen\xdim \newdimen\ydim \newdimen\PfadD@cke \newdimen\Pfadd@cke
\def\PfadDicke#1{\PfadD@cke#1 \divide\PfadD@cke by2 \Pfadd@cke\PfadD@cke \multiply\PfadD@cke by2}
\long\def\LOOP#1\REPEAT{\def\BODY{#1}\ITERATE}
\def\ITERATE{\BODY \let\next\ITERATE \else\let\next\relax\fi \next}
\let\REPEAT=\fi
\def\Punkt{\hbox{\raise-2pt\hbox to0pt{\hss\scriptsize$\bullet$\hss}}}
\def\DuennPunkt(#1,#2){\unskip
  \raise#2 \Einheit\hbox to0pt{\hskip#1 \Einheit
          \raise-2.5pt\hbox to0pt{\hss\normalsize$\bullet$\hss}\hss}}
\def\NormalPunkt(#1,#2){\unskip
  \raise#2 \Einheit\hbox to0pt{\hskip#1 \Einheit
          \raise-3pt\hbox to0pt{\hss\large$\bullet$\hss}\hss}}
\def\DickPunkt(#1,#2){\unskip
  \raise#2 \Einheit\hbox to0pt{\hskip#1 \Einheit
          \raise-4pt\hbox to0pt{\hss\Large$\bullet$\hss}\hss}}
\def\Kreis(#1,#2){\unskip
  \raise#2 \Einheit\hbox to0pt{\hskip#1 \Einheit
          \raise-4pt\hbox to0pt{\hss\Large$\circ$\hss}\hss}}
\def\Diagonale(#1,#2)#3{\unskip\leavevmode
  \xcoord#1\relax \ycoord#2\relax
      \raise\ycoord \Einheit\hbox to0pt{\hskip\xcoord \Einheit
         \unitlength\Einheit
         \line(1,1){#3}\hss}}
\def\AntiDiagonale(#1,#2)#3{\unskip\leavevmode
  \xcoord#1\relax \ycoord#2\relax 
      \raise\ycoord \Einheit\hbox to0pt{\hskip\xcoord \Einheit
         \unitlength\Einheit
\tile         \line(1,-1){#3}\hss}}
\def\Pfad(#1,#2),#3\endPfad{\unskip\leavevmode
  \xcoord#1 \ycoord#2 \thicklines\ZeichnePfad#3\endPfad\thinlines}
\def\ZeichnePfad#1{\ifx#1\endPfad\let\next\relax
  \else\let\next\ZeichnePfad
    \ifnum#1=1
      \raise\ycoord \Einheit\hbox to0pt{\hskip\xcoord \Einheit
         \vrule height\Pfadd@cke width1 \Einheit depth\Pfadd@cke\hss}%
      \advance\xcoord by 1
    \else\ifnum#1=2
      \raise\ycoord \Einheit\hbox to0pt{\hskip\xcoord \Einheit
        \hbox{\hskip-\PfadD@cke\vrule height1 \Einheit width\PfadD@cke depth0pt}\hss}%
      \advance\ycoord by 1
    \else\ifnum#1=3
      \raise\ycoord \Einheit\hbox to0pt{\hskip\xcoord \Einheit
         \unitlength\Einheit
         \line(1,1){1}\hss}
      \advance\xcoord by 1
      \advance\ycoord by 1
    \else\ifnum#1=4
      \raise\ycoord \Einheit\hbox to0pt{\hskip\xcoord \Einheit
         \unitlength\Einheit
         \line(1,-1){1}\hss}
      \advance\xcoord by 1
      \advance\ycoord by -1
    \else\ifnum#1=5
      \advance\xcoord by -1
      \raise\ycoord \Einheit\hbox to0pt{\hskip\xcoord \Einheit
         \vrule height\Pfadd@cke width1 \Einheit depth\Pfadd@cke\hss}%
    \else\ifnum#1=6
      \advance\ycoord by -1
      \raise\ycoord \Einheit\hbox to0pt{\hskip\xcoord \Einheit
        \hbox{\hskip-\PfadD@cke\vrule height1 \Einheit width\PfadD@cke depth0pt}\hss}%
    \else\ifnum#1=7
      \advance\xcoord by -1
      \advance\ycoord by -1
      \raise\ycoord \Einheit\hbox to0pt{\hskip\xcoord \Einheit
         \unitlength\Einheit
         \line(1,1){1}\hss}
    \else\ifnum#1=8
      \advance\xcoord by -1
      \advance\ycoord by +1
      \raise\ycoord \Einheit\hbox to0pt{\hskip\xcoord \Einheit
         \unitlength\Einheit
         \line(1,-1){1}\hss}
    \fi\fi\fi\fi
    \fi\fi\fi\fi
  \fi\next}
\def\hSSchritt{\leavevmode\raise-.4pt\hbox to0pt{\hss.\hss}\hskip.2\Einheit
  \raise-.4pt\hbox to0pt{\hss.\hss}\hskip.2\Einheit
  \raise-.4pt\hbox to0pt{\hss.\hss}\hskip.2\Einheit
  \raise-.4pt\hbox to0pt{\hss.\hss}\hskip.2\Einheit
  \raise-.4pt\hbox to0pt{\hss.\hss}\hskip.2\Einheit}
\def\vSSchritt{\vbox{\baselineskip.2\Einheit\lineskiplimit0pt
\hbox{.}\hbox{.}\hbox{.}\hbox{.}\hbox{.}}}
\def\DSSchritt{\leavevmode\raise-.4pt\hbox to0pt{%
  \hbox to0pt{\hss.\hss}\hskip.2\Einheit
  \raise.2\Einheit\hbox to0pt{\hss.\hss}\hskip.2\Einheit
  \raise.4\Einheit\hbox to0pt{\hss.\hss}\hskip.2\Einheit
  \raise.6\Einheit\hbox to0pt{\hss.\hss}\hskip.2\Einheit
  \raise.8\Einheit\hbox to0pt{\hss.\hss}\hss}}
\def\dSSchritt{\leavevmode\raise-.4pt\hbox to0pt{%
  \hbox to0pt{\hss.\hss}\hskip.2\Einheit
  \raise-.2\Einheit\hbox to0pt{\hss.\hss}\hskip.2\Einheit
  \raise-.4\Einheit\hbox to0pt{\hss.\hss}\hskip.2\Einheit
  \raise-.6\Einheit\hbox to0pt{\hss.\hss}\hskip.2\Einheit
  \raise-.8\Einheit\hbox to0pt{\hss.\hss}\hss}}
\def\SPfad(#1,#2),#3\endSPfad{\unskip\leavevmode
  \xcoord#1 \ycoord#2 \ZeichneSPfad#3\endSPfad}
\def\ZeichneSPfad#1{\ifx#1\endSPfad\let\next\relax
  \else\let\next\ZeichneSPfad
    \ifnum#1=1
      \raise\ycoord \Einheit\hbox to0pt{\hskip\xcoord \Einheit
         \hSSchritt\hss}%
      \advance\xcoord by 1
    \else\ifnum#1=2
      \raise\ycoord \Einheit\hbox to0pt{\hskip\xcoord \Einheit
        \hbox{\hskip-2pt \vSSchritt}\hss}%
      \advance\ycoord by 1
    \else\ifnum#1=3
      \raise\ycoord \Einheit\hbox to0pt{\hskip\xcoord \Einheit
         \DSSchritt\hss}
      \advance\xcoord by 1
      \advance\ycoord by 1
    \else\ifnum#1=4
      \raise\ycoord \Einheit\hbox to0pt{\hskip\xcoord \Einheit
         \dSSchritt\hss}
      \advance\xcoord by 1
      \advance\ycoord by -1
    \else\ifnum#1=5
      \advance\xcoord by -1
      \raise\ycoord \Einheit\hbox to0pt{\hskip\xcoord \Einheit
         \hSSchritt\hss}%
    \else\ifnum#1=6
      \advance\ycoord by -1
      \raise\ycoord \Einheit\hbox to0pt{\hskip\xcoord \Einheit
        \hbox{\hskip-2pt \vSSchritt}\hss}%
    \else\ifnum#1=7
      \advance\xcoord by -1
      \advance\ycoord by -1
      \raise\ycoord \Einheit\hbox to0pt{\hskip\xcoord \Einheit
         \DSSchritt\hss}
    \else\ifnum#1=8
      \advance\xcoord by -1
      \advance\ycoord by 1
      \raise\ycoord \Einheit\hbox to0pt{\hskip\xcoord \Einheit
         \dSSchritt\hss}
    \fi\fi\fi\fi
    \fi\fi\fi\fi
  \fi\next}
\def\Koordinatenachsen(#1,#2){\unskip
 \hbox to0pt{\hskip-.5pt\vrule height#2 \Einheit width.5pt depth1 \Einheit}%
 \hbox to0pt{\hskip-1 \Einheit \xcoord#1 \advance\xcoord by1
    \vrule height0.25pt width\xcoord \Einheit depth0.25pt\hss}}
\def\Koordinatenachsen(#1,#2)(#3,#4){\unskip
 \hbox to0pt{\hskip-.5pt \ycoord-#4 \advance\ycoord by1
    \vrule height#2 \Einheit width.5pt depth\ycoord \Einheit}%
 \hbox to0pt{\hskip-1 \Einheit \hskip#3\Einheit 
    \xcoord#1 \advance\xcoord by1 \advance\xcoord by-#3 
    \vrule height0.25pt width\xcoord \Einheit depth0.25pt\hss}}
\def\Gitter(#1,#2){\unskip \xcoord0 \ycoord0 \leavevmode
  \LOOP\ifnum\ycoord<#2
    \loop\ifnum\xcoord<#1
      \raise\ycoord \Einheit\hbox to0pt{\hskip\xcoord \Einheit\Punkt\hss}%
      \advance\xcoord by1
    \repeat
    \xcoord0
    \advance\ycoord by1
  \REPEAT}
\def\Gitter(#1,#2)(#3,#4){\unskip \xcoord#3 \ycoord#4 \leavevmode
  \LOOP\ifnum\ycoord<#2
    \loop\ifnum\xcoord<#1
      \raise\ycoord \Einheit\hbox to0pt{\hskip\xcoord \Einheit\Punkt\hss}%
      \advance\xcoord by1
    \repeat
    \xcoord#3
    \advance\ycoord by1
  \REPEAT}
\def\Label#1#2(#3,#4){\unskip \xdim#3 \Einheit \ydim#4 \Einheit
  \def\lo{\advance\xdim by-.5 \Einheit \advance\ydim by.5 \Einheit}%
  \def\llo{\advance\xdim by-.25cm \advance\ydim by.5 \Einheit}%
  \def\loo{\advance\xdim by-.5 \Einheit \advance\ydim by.25cm}%
  \def\o{\advance\ydim by.25cm}%
  \def\ro{\advance\xdim by.5 \Einheit \advance\ydim by.5 \Einheit}%
  \def\rro{\advance\xdim by.25cm \advance\ydim by.5 \Einheit}%
  \def\roo{\advance\xdim by.5 \Einheit \advance\ydim by.25cm}%
  \def\l{\advance\xdim by-.30cm}%
  \def\r{\advance\xdim by.30cm}%
  \def\lu{\advance\xdim by-.5 \Einheit \advance\ydim by-.6 \Einheit}%
  \def\llu{\advance\xdim by-.25cm \advance\ydim by-.6 \Einheit}%
  \def\luu{\advance\xdim by-.5 \Einheit \advance\ydim by-.30cm}%
  \def\u{\advance\ydim by-.30cm}%
  \def\ru{\advance\xdim by.5 \Einheit \advance\ydim by-.6 \Einheit}%
  \def\rru{\advance\xdim by.25cm \advance\ydim by-.6 \Einheit}%
  \def\ruu{\advance\xdim by.5 \Einheit \advance\ydim by-.30cm}%
  #1\raise\ydim\hbox to0pt{\hskip\xdim
     \vbox to0pt{\vss\hbox to0pt{\hss$#2$\hss}\vss}\hss}%
}
\catcode`\@=12

\begin{document}

\title{Yay for determinants!}


\author[T. Amdeberhan]{Tewodros Amdeberhan}
\address{Department of Mathematics,
Tulane University, New Orleans, LA 70118, USA}
\email{tamdeber@tulane.edu}
\author[C. Koutschan]{Christoph Koutschan}
\address{Johann Radon Institute for Computational and Applied Mathematics, Austrian Academy of Sciences, Altenberger Strasse 69, 4040 Linz, Austria}
\email{christoph.koutschan@oeaw.ac.at}
\author[D. Zeilberger]{Doron Zeilberger}
\address{Department of Mathematics, Rutgers University (New Brunswick), \\
Hill Center-Busch Campus, 110 Frelinghuysen Rd., Piscataway, NJ 08854-8019, USA}
\email{DoronZeil@gmail.com}

\begin{abstract} In this {\it case study}, we hope to show why Sheldon Axler was not just wrong, but {\em wrong}, when he urged, in 1995:  ``Down with Determinants!''.
We first recall how determinants are useful in enumerative combinatorics, and then illustrate
three versatile tools (Dodgson's condensation, the holonomic ansatz and constant term evaluations) to operate in tandem to prove a certain intriguing determinantal formula conjectured by the first author. 
We conclude with a postscript describing yet another, much more efficient, method for evaluating determinants: `ask determinant-guru, Christian Krattenthaler', but advise people
only to use it as a last resort, since if we would have used this last method right away, we would not have had the fun of doing it all by ourselves.
\end{abstract}

\maketitle
{\em Dedicated to Maestro\footnote{in piano, determinants, and
  hypergeometrics, inter alia.}
 Christian Krattenthaler
(b. 8 Oct., 1958) on his 

turning millionth-and-one years old.}

\bigskip
\noindent
{\em ``Philosophers and psychiatrists should explain why it is that we mathematicians are in the habit of systematically erasing our footsteps. Scientists have always looked askance at this strange habit of mathematicians, which has changed little from Pythagoras to our day}".
--Gian-Carlo Rota.\footnote{Two Turning Points in Invariant Theory, Math. Intell. 21(1) (Winter 1999), p. 26.}

\section{The Joy of Determinants}

\smallskip
\noindent
In 1995, Sheldon Axler published an article \cite{Axler}  with a very catchy title (that being determinantal lovers, we {\it love to hate}): `Down with Determinants!' 
This admittedly well-written paper was well received and even won the prestigious Lester Ford award for that year. Let's quote the first paragraph.

{\it
``Ask anyone why a square matrix of complex numbers has an eigenvalue, 

and you'll probably get the wrong answer, which goes something like this: 

The characteristic polynomial of the matrix---which is defined via 

determinants---has a root (by the fundamental theorem of algebra); 

this root is an eigenvalue of the matrix. What's wrong with that answer? 

It depends upon determinants, that's what. Determinants are difficult, 

non-intuitive, and often defined without motivation. As we'll see, 

there is a better proof---one that is simpler, clearer,  provides more insight, 

and avoids determinants.''}

\smallskip
\noindent
Axler then goes on to describe a {\it determinant-less} approach to linear algebra, that while very elegant, is {\bf not} to our liking. It is way too {\it abstract} for our taste.
We believe that one of the ills of undergraduate mathematics instruction is  excess abstraction. We love determinants because they are so {\bf concrete}. 

\smallskip
\noindent
In our humble opinion, determinants are {\bf easy}, {\bf intuitive}, and can be easily defined with great motivation. In fact, it is a straightforward extension of
the good-old factorial function $n!$. Recall that $n!$ {\bf counts} the number of permutations of the set $\{1, \dots, n\}$. In other words
\begin{align*}
n! = \sum_{\pi \in S_n} 1 \; .
\end{align*}
Now for any permutation $\pi=(\pi_1, \dots, \pi_n)$ define a {\it weight}
\begin{align*}
\operatorname{Weight}(\pi):&= \operatorname{sign}(\pi) \, a_{1,\pi_1} \cdots a_{n,\pi_n} \;,
\end{align*}
and then the determinant is {\it simply} the weight-enumerator of the symmetric group:
\begin{align*}
\det(a_{ij}) = \sum_{\pi \in S_n} \operatorname{Weight}(\pi) \; .
\end{align*}

\smallskip
\noindent
Note that our definition is {\bf combinatorial}, and indeed they come up so often in tough enumeration problems, where it is already a challenge
to express the quantity of interest as an {\bf explicit} (symbolic) determinant, and then an often bigger challenge is to {\bf evaluate} it
in {\bf closed form}. In this {\it case study} we will discuss various ways to tackle such determinants, by focusing on one `hard nut' that came up in MathOverflow.

\section{An Intriguing Determinant}

\noindent
The discussion in this  article, and irs rendition, is motivated by a question on MathOverflow \cite{Cigler} posed by Johann Cigler
wherein he asked for a simple direct proof of the determinant
$$\det A_{n,m}:=\det \left[ \binom{2m}{j-i+m}-\binom{2m}{m-i-j-1} \right]_{i,j=0}^{n-1}
=\prod_{i\leq j}^{1,m-1}\frac{2n+i+j}{i+j}.$$
The expression on the right-hand side of the identity is connected to a host of interesting combinatorial interpretations (see references \cite{Ham}, \cite{Jonsson}, \cite{Kratt}, \cite{Proctor} and \cite{Stanley}).

\smallskip
\noindent
Our journey began with a generalization (true to P\'{o}lya's dictum: {\em generalize to trivialize}). Introduce a new parameter $x$ and a matrix given by
\begin{align*}
T_{n,m}(x):&=\left[\binom{x+m}{j-i+m}-\binom{x+m}{m-i-j-1}\right]_{i,j=0}^{n-1}.
\end{align*}
After enough experimentation, we managed to guess a closed-form evaluation
for its determinant.

\noindent
\begin{conjecture} \label{conj1} We have
$$\det T_{n,m}(x)=\prod_{i=1}^n\prod_{j=1}^m\frac{(x+i-j)(x+2i+j-2)}{(x+2i-j)(i+j-1)}.$$
\end{conjecture}

\noindent
None of our initial attempts at proving this claim worked. 

\begin{remark} Meanwhile, we made the following curious observation. It requires introducing one notion. Let's denote by $CT_y\,F(y)$ the {\em constant term} of a Laurent polynomial $F(y)$. For example, if $F(y)=(3y^{-2}-2y^{-1}+5y+1)(y^3+6y+y^{-1})$ then to compute its constant term, expand fully as a Laurent series
$$F(y)=\frac3{y^3}-\frac2{y^2}+\frac{19}y-7+9y+28y^2+y^3+5y^4.$$
Therefore $CT_yF(y)=-7$.

\smallskip
\noindent
With this in mind, define the function
$$\varphi_n(y):=\frac{(1+y^n-y^{n+1}+y^{2n})(1-y^n)^2}{y^{n-1}(1+y)(1-y)^2}\in \mathbb{Z}[y,y^{-1}]$$
then
\begin{align*}
\det T_{n,n}(x)=n!^{-n}\prod_{i=0}^{n-1}\binom{n+i}i^{-1} \cdot \prod_{s=1-n}^{3n-2}(x+s)^{CT_y(y^{-s}\varphi_n(y))}.
\end{align*}

\end{remark}

\section{Enter Dodgson}

\noindent
A relatively simple choice for conjecturing and proving 
explicit determinant evaluations 
is  a method  inspired by {\em Dodgson's condensation}~\cite{Zeil1}, that goes back to Jacobi and Desnanot, and popularized by the Reverend Charles L. Dodgson (aka Lewis Carroll). This is the third author's favorite method (see Section 2.3 of~\cite{Kratt2}).
However, $\det T_{n,m}(x)$ did not yield to this at all. It was necessary to change the approach. Our first breakthrough was realized in the form of discovering a new matrix
\begin{align*}
D_n(a,b):&=\left[\binom{2i+2a}{i-j+a-b}-\binom{2i+2a}{i-j+a-b-1}\right]_{i,j=0}^{n-1}
\end{align*}
which, happily, is amenable to the present technique.

\begin{theorem} \label{dodg1} We have the determinantal evaluation
\begin{align*}
\det D_n(a,b)=\prod_{i=1}^n\prod_{j=1}^{a-b}\frac{(a+b+i-j)(a+b+2i+j-2)}{(a+b+2i-j)(i+j-1)}.
\end{align*}
\end{theorem}
\begin{proof} Denoting the determinant by $M_n(a,b)$, Dodgson's condensation states that
$$M_n(a,b)=\frac{M_{n-1}(a,b)M_{n-1}(a+1,b+1)-M_{n-1}(a+1,b)M_{n-1}(a,b+1)}{M_{n-2}(a+1,b+1)}.$$
Therefore, we only need to verify that the right-hand side does obey the same recurrence. Then, check both sides agree, say, when $n=1$ and $n=2$.
\end{proof}

\noindent
For reasons that will become clearer soon, let's reformulate the above result.

\begin{corollary} \label{after-dodg1} We have
\begin{align*}
\det B_{n,m}(x):&=\det \left[\binom{x+m+2i}{i-j+m}-\binom{x+m+2i}{i-j+m-1}\right]_{i,j=0}^{n-1} \\
&=\prod_{i=1}^n\prod_{j=1}^m\frac{(x+i-j)(x+2i+j-2)}{(x+2i-j)(i+j-1)}.
\end{align*}
\end{corollary}
\begin{proof} This follows from Theorem \ref{dodg1}  with $a=\frac{x+m}2$ and $b=\frac{x-m}2$ in $D_n(a,b)$.
\end{proof}

\noindent
Note that the {\it right hand sides} of the (already proved!) Corollary \ref{after-dodg1} and the (still unproved) Conjecture \ref{conj1} are the same, but why is it useful?
We will see in the next section that the {\bf left hand sides} are the same, and, hence Conjecture 1 would follow from the {\it transitivity} of the $=$ relation:
$A=B$ and $B=C$ $\Rightarrow$ $A=C$.

\smallskip
\noindent
This method was famously used by one of us (DZ) to prove the Mills-Robbins-Rumsey {\it alternating sign matrix conjecture} \cite{Zeil5} (see also \cite{Bressoud}).
In that article the author first found a (complicated) constant term expression for the desired quantity (the number of alternating sign matrices) and {\it another} (almost as complicated) constant term expression for
another quantity ({\it totally symmetric self-complementary plane partitions}), already proved by guru George Andrews to be equal to the desired expression.
He then worked very hard to prove that these two constant term expressions are the same, and finally,
by taking a {\em free ride}  from Andrews' previously proved enumeration, and using the above `{\it transitivity} of $=$', gave the {\bf first} proof of this notorious conjecture.

\smallskip
\noindent
In the next section we will use this methodology, but things will be much simpler, since we will prove that the left sides of Conjecture \ref{conj1} and Corollary \ref{after-dodg1} are given by the
{\it same} constant-term expression, so unlike \cite{Zeil5}, we will not have to work hard to prove that $A=B$, since $A$ is {\bf exactly} $B$.

\smallskip
\section{Enter Constant Term Evaluations}

\noindent
As promised at the end of the last section, in this section, we describe the connection between the two matrices $B_{n,m}(x)$ and $T_{n,m}(x)$ from the preceding sections.

\begin{theorem} \label{CT1} It holds true that $\det T_{n,m}(x)=\det B_{n,m}(x)$, and hence so does Conjecture \ref{conj1}.
\end{theorem}
\begin{proof} Denote the {\em constant term} of the Laurent series expansion of $F(t)$ by $CT_{t}F$. In view of this, the entries of $T_{n,m}(x)$ can be expressed as a constant term:
$$\binom{x+m}{j-i+m}-\binom{x+m}{m-i-j-1} =
  CT_t  \left(1+\frac1t\right)^{x+m} t^m(t^{j-i} -t^{-i-j-1}).$$
The entries in the $i^{\text{th}}$ row of the matrix $T_{n,m}(x)$ are (we need a different variable for each row, let's call it $t_i$ - this does not change the values)
$$CT_{t_i} \left(1+\frac1{t_i}\right)^{x+m} t_i^m(t_i^{j-i} -t_i^{-i-j-1}).$$
Let $\vec{t}=(t_0,t_1,\dots,t_{n-1})$. We apply multi-linearity of determinants to obtain
\begin{align*}
\det T_{n,m}(x)&=\det\left[CT_{t_i} \left(1+\frac1{t_i}\right)^{x+m} t_i^{m-i}\left(t_i^{j} -t_i^{-j-1}\right)\right] \\
&= CT_{\vec{t}} \left[\prod_{i=0}^{n-1} \left(1+t_i\right)^{x+m}\cdot t_i^{-x-i-\frac12} \cdot \det \left(t_i^{j+\frac12} -t_i^{-j-\frac12}\right)\right] \\
&=CT_{\vec{t}}\left[\prod_{i=0}^{n-1}\frac{(1+t_i)^{x+m}(t_i-1)}{t_i^{x+n+i}}\cdot \prod_{i<j}^{0,n-1}(t_i-t_j)(1-t_it_j)\right];
\end{align*}
where we have utilized \cite[Lemma~2]{Kratt2}.
\noindent
Let us {\em symmetrize} the right-hand side (by averaging over the images of the symmetric group $\mathfrak{S}_n$, using the fact that the functional {\em CT} is unaffected by permuting the variables,
since {\bf a constant is always a constant}; this was called the {\em Stanton-Stembridge trick} in \cite{Zeil5}).
Note that only the product of two factors 
$\prod_{i=0}^{n-1}t_i^{-i}\cdot\prod_{i<j}(t_i-t_j)$ will be impacted to introduce a new Vandermonde:
\begin{align*}
&\det T_{n,m}(x) \\
&=\frac1{n!}\, 
CT_{\vec{t}} \left\{\prod_{i=0}^{n-1}\left[ \frac{(1+t_i)^{x+m}(t_i-1)}{t_i^{x+n}}\right] \cdot
\prod_{i<j}^{0,n-1}(t_i-t_j)(t_j^{-1}-t_i^{-1})(1-t_it_j)\right\} \\
&=\frac1{n!}\,
CT_{\vec{t}} \left\{\prod_{i=0}^{n-1}\left[ \frac{(1+t_i)^{x+m}(t_i-1)}{t_i^{x+2n-1}}\right] \cdot
\prod_{i<j}^{0,n-1}(t_i-t_j)^2(1-t_it_j)\right\}.
\end{align*}

\smallskip
\noindent
Now, repeat the constant term extraction on the determinant from the previous section. The outcome of this procedure is:

\begin{align*}
&\det B_{n,m}(x) \\
&=\det\left[CT_{t_i} \left(1+\frac1{t_i}\right)^{x+m+2i} t_i^m(t_i^{i-j} -t_i^{i-j-1})\right] \\
&=CT_{\vec{t}} \left\{\prod_{i=0}^{n-1}\left[\left(1+\frac1{t_i}\right)^{x+m+2i}t_i^{m+i-1}(t_i-1)\right]\prod_{i<j}^{0,n-1}(t_j^{-1}-t_i^{-1})\right\} \\
&=CT_{\vec{t}} \left\{\prod_{i=0}^{n-1}\left[\frac{(1+t_i)^{x+m}(t_i-1)}{t_i^{x+n}}\right](2+t_i+t_i^{-1})^i\prod_{i<j}^{0,n-1}(t_i-t_j)\right\} \\
&=\frac1{n!}\,CT_{\vec{t}} \left\{\prod_{i=0}^{n-1}\left[\frac{(1+t_i)^{x+m}(t_i-1)}{t_i^{x+n}}\right] \prod_{i<j}^{0,n-1}(t_i-t_j)\left(t_j+t_j^{-1}-t_i-t_i^{-1}\right)\right\} \\
&=\frac1{n!}\,CT_{\vec{t}} \left\{\prod_{i=0}^{n-1}\left[\frac{(1+t_i)^{x+m}(t_i-1)}{t_i^{x+2n-1}}\right] \prod_{i<j}^{0,n-1}(t_i-t_j)^2(1-t_it_j)\right\}.
\end{align*}

\noindent
Based on our analysis shown above, both determinants $\det T_{n,m}(x)$ and $\det B_{n,m}(x)$ are equal to the same constant term
$$\frac1{n!}\,CT_{\vec{t}} \left\{\prod_{i=0}^{n-1}\left[(1+t_i)^{x+m}(t_i-1)t_i^{-x-2n+1}\right] \prod_{i<j}^{0,n-1}(t_i-t_j)^2(1-t_it_j)\right\}.$$
Hence $\det T_{n,m}(x)$ and $\det B_{n,m}(x)$ must be equal to each other, as claimed.
\end{proof}

\noindent
In particular, we thus recover the determinant of Cigler as stated in Section 2:
\begin{corollary} {\rm(}The case $x=m${\rm)} We have
\begin{align*}
\det\left[\binom{2m}{j-i+m}-\binom{2m}{m-i-j-1}\right]_{i,j=0}^{n-1}=\prod_{1\leq i\leq j\leq m-1}\frac{2n+i+j}{i+j}.
\end{align*}
\end{corollary}
\begin{proof} The specialization $x=m$ in the formula for $\det T_{n,m}(x)$ results in
$$\det T_{n,m}(m)=\prod_{1\leq i\leq j\leq m-1}\frac{2n+i+j}{i+j}$$
after some simplification, which is $\det A_{n,m}$.
\end{proof}
\noindent
In addition, we believe to have uncovered a {\em new} constant term evaluation which appears not to fit any of the existing formulas in the literature.
\begin{corollary} We have
\begin{align*}
&CT_{\vec{t}} \left\{\prod_{i=0}^{n-1}\left[\frac{(1+t_i)^{x+m}(t_i-1)}{t_i^{x+2n-1}}\right] \prod_{i<j}^{0,n-1}(t_i-t_j)^2(1-t_it_j)\right\}  \\
&=n!\cdot \prod_{i=1}^n\prod_{j=1}^m\frac{(x+i-j)(x+2i+j-2)}{(x+2i-j)(i+j-1)}.
\end{align*}
\end{corollary}
\begin{proof} This is a consequence of Corollary \ref{after-dodg1} and the proof of Theorem \ref{CT1}.
\end{proof}

\section{Enter the Holonomic Ansatz}

\noindent
In this section, we use an automated method called the {\em holonomic ansatz}
due to the third author~\cite{Zeil3}, and extended by the second author and
collaborators~\cite{Du,Koutschan1,Koutschan3}
to discover and prove many new deep determinant identities,
which will provide yet another proof of Conjecture~1.
\smallskip
\noindent
Note: we also reindex $1\leq i,j \leq n$ instead of $0\leq i,j \leq n-1$. Let us recall the matrix $A_n:=A_n(m,x):=(a_{i,j})_{1\leq i,j\leq n}$
whose entries $a_{i,j}$ are given as follows:
\[
  a_{i,j} := a_{i,j}(m,x) := \binom{m+x}{m-i+j}-\binom{m+x}{m-i-j+1}.
\]
Denote by $A_n^{(i,j)}$ the matrix $A_n$ with the $i^{\text{th}}$-row and the $j^{\text{th}}$-column being removed. Then we define
\begin{equation}\label{eq:def_cnj}
  c_{n,j} := c_{n,j}(m,x) := (-1)^{n+j}\,\frac{\det A_n^{(n,j)}}{\det A_{n-1}}.
\end{equation}
Employing an ansatz with undetermined coefficients (in other words, guessing), we
find plausible recurrence relations that are conjecturally satisfied
by the sequence $c_{n,j}$, which suggest that it is holonomic with rank~$2$:
\begin{align*}
  & j (j-n-1) (m+n) (2 n+x-1) (2 n+x) c_{n+1,j} +\\
  & (2 j-1) n (j-m) (2 n+x) (j-n-x+1) c_{n,j+1}+ \\
  & n (j+n+x-1) (4j^2n-(m^2 + 4n^2 - m)j \\
& \qquad -2mn+(2j^2 - 4jn - m)x - jx^2) c_{n,j} = 0, \qquad \text{and} \\
  & j (j-m+1) (j+n+1) (j-n-x+2) c_{n,j+2} \\
  & +(2 j^4+4 j^3-j^2 m^2-2 j^2 n^2-2 j^2 n x+2 j^2 n-j^2 x^2+2 j^2 x+2 j^2-j m^2 \\
  & \quad -2 j n^2-2 j n x+2 j n-j x^2+2 j x-m n^2-m n x+m n) c_{n,j+1} \\
  & +(j+1) (j+m) (j-n) (j+n+x-1) c_{n,j} = 0.
\end{align*}

\smallskip
\noindent
For the purpose of guessing, one has to evaluate $c_{n,j}$ for concrete
values of $n$ and~$j$; in this particular example it suffices to consider
$1\leq j\leq n\leq 15$. For a concrete integer~$n$, the vector
$(c_{n,j})_{1\leq j\leq n}$ can be determined by computing the
kernel of the matrix $(a_{i,j})_{1\leq i<n,1\leq j\leq n}$. Note
that this kernel computation appears to be challenging, because the matrix
entries are not rational functions in the parameters $m$ and~$x$. To overcome
this problem, and also for efficiency reasons, we employ the
evaluation-interpolation technique: when $m$ and~$x$ are substituted by
concrete integers, then the matrix entries turn into integers, and the
recurrences for~$c_{n,j}$ can be easily guessed. Performing the same
computation for different choices of $m$ and~$x$ allows one to reconstruct
recurrences with symbolic $m$ and~$x$, as given above.

\smallskip
\noindent
The shape of the above recurrences, i.e., their support and their leading
coefficients, implies that the following three initial conditions are
sufficient to define a unique bivariate sequence: $c_{1,1}=1$, $c_{1,2}=0$,
$c_{2,2}=1$. From now on, let $c_{n,j}$ denote this unique bivariate sequence
that is defined by unrolling the recurrences, starting from the given initial
conditions. We want to prove that this conjectured definition of $c_{n,j}$
agrees with its original definition~\eqref{eq:def_cnj}. For this purpose
it suffices to establish the two identities:
\begin{align}
  c_{n,n} &= 1 \qquad (n\geq1), \label{eq:id1} \\
  \sum_{j=1}^n a_{i,j} c_{n,j} &= 0 \qquad (1\leq i<n), \label{eq:id2}
\end{align}
which can be achieved by holonomic closure properties and creative
telescoping, respectively. We use the Mathematica package
HolonomicFunctions~\cite{Koutschan2}. Observe that the closure property ``integer-linear
substitution'' implies that $c_{n,n}$ satisfies a univariate recurrence of
order at most~$2$. We compute this recurrence and show that it admits a
constant solution. Then the two initial values $c_{1,1}=c_{2,2}=1$ imply
that $c_{n,n}=1$ for all $n\geq1$.

\smallskip
\noindent
In order to prove \eqref{eq:id2}, we split the sum into two parts:
\[
  s_{i,n} := \sum_{j=1}^n a_{i,j} c_{n,j} =
  \sum_{j=1}^n \binom{m+x}{m-i+j} c_{n,j} -
  \sum_{j=1}^n \binom{m+x}{m-i-j+1} c_{n,j}.
\]
Now creative telescoping can be employed to derive a set of recurrences for
each of the sums on the right-hand side, which can be combined to a set of
recurrences that is satisfied by the whole expression~$s_{i,n}$.  These
computations are far from trivial (it took about two hours of CPU time) and
the recurrences are too big to be displayed here. It can easily be checked
that $s_{i,n}=0$ for a few concrete small integers $i$ and~$n$, and the
computed recurrences then imply that $s_{i,n}=0$ for all $1\leq i<n$.

\smallskip
\noindent
Finally, we investigate the sum
\[
  t_n := t_n(m,x) := \sum_{j=1}^n a_{n,j} c_{n,j} =
  \sum_{j=1}^n \binom{m+x}{m-n+j} c_{n,j} -
  \sum_{j=1}^n \binom{m+x}{m-n-j+1} c_{n,j}.
\]

\smallskip
\noindent
Another two applications of creative telescoping exhibit that both sums on the
right-hand side satisfy the same recurrence, which therefore is also valid for
their difference $t_n$:
\begin{align*}
  & (m+n) (2 n+x-1) (2 n+x+1) (2 n+x)^2 (m-n-x) t_{n+1} \\
  & +n (n+x) (m-2 n-x-1) (m-2 n-x) (m+2 n+x-1) (m+2 n+x) t_n = 0.
\end{align*}
Since this is a first-order recurrence, a closed form for $t_n$ is
immediately obtained in terms of {\em Euler's gamma function}:
\[
  t_n = \frac{\Gamma(n)\,\Gamma(n+x)\,\Gamma(2n+x-m)\,\Gamma (2n+x+m-1)}%
  {\Gamma(n+m)\,\Gamma(n+x-m)\,\Gamma(2n+x)\,\Gamma(2n+x-1)}.
\]
From the definition \eqref{eq:def_cnj} of $c_{n,j}$ it follows that
\[
  t_n = \frac{\det A_n}{\det A_{n-1}},
\]
which yields the desired closed form of our determinant:
\[
  \det A_n = \prod_{i=1}^n t_i =
  \prod_{i=1}^n \prod_{j=1}^m \frac{t_i(j,x)}{t_i(j-1,x)} =
  \prod_{i=1}^n \prod_{j=1}^m \frac{(x+i-j)(x+2i+j-2)}{(x+2i-j)(i+j-1)}. \qquad \square
\]

\noindent
{\bf The special case $x=m$:} we point out something very pleasant happening here. The entries $a_{i,j}$ of the matrix $T_{n,m}(m)$ are given by
$$a_{i,j}=\binom{2m}{m-i+j}-\binom{2m}{m-i-j+1},$$
and the corresponding determinant by
$$b_n:=\det T_{n,m}(m)=\prod_{1\leq i\leq j\leq m-1}\frac{2n+i+j}{i+j}.$$
The key difference: this time, we are able to construct an {\em explicit} doubly-indexed sequence $\mathbf{c}$ according to $c_{n,n}=1$ and for $1\leq j<n$,
$$c_{n,j}=\frac{(-1)^{n-j}\binom{2n-1}{n-j}\binom{n+m-j-1}{n-j}}{\binom{2n+m-2}{n-j}}.$$
Then, Zeilberger's holonomic ansatz rules supreme once the following system of equations are proven to hold true:
\begin{align*}
\begin{cases}
\qquad \qquad \,\,\, \,\, c_{n,n}=1 \qquad \qquad n\geq1 \\
\sum_{j=1}^na_{i,j}\,c_{n,j}\,=0 \qquad \qquad 1\leq i\leq n-1 \\
\sum_{j=1}^na_{n,j}\,c_{n,j}=\frac{b_n}{b_{n-1}} \qquad \,\, n\geq1.
\end{cases}
\end{align*}
But, these are directly justified with the help of {\em Zeilberger's algorithm} \cite{Zeil4}. We will not pursue this matter because we have already proven Conjecture 1 (see above), in its 
full generality.

\section{$T_{n,m}(x)$ Finds its Twin}

\noindent
Suppressing the parameters $a$ and $b$, and defining $u_k:=\binom{a+b}{a-k}$, the matrix $T_{c,a}(b)$ is nothing but the difference of a Toeplitz and a Hankel matrix:
$$T_{c,a}(b)=[u_{i-j}-u_{i+j+1}].$$
By analogy, we may introduce its natural companion as a sum of a Toeplitz and a Hankel matrix:
$$\widetilde{T}_{c,a}(b)=[u_{i-j}+u_{i+j+1}].$$
A truly parallel argument (as in the previous sections) produces a very similar determinantal evaluation:

\begin{theorem} We have
\begin{align*} \det \widetilde{T}_{c,a}(b)=\prod_{i=1}^c\prod_{j=1}^a\frac{(b+i-j)(b+2i+j-1)}{(b+2i-j-1)(i+j-1)}.
\end{align*}
\end{theorem}

\noindent
To give some context, let $f(z)$ be a complex $L_1(\mathcal{C})$ function on the unit circle having Fourier coefficients 
$$u_k=\frac1{2\pi}\int_0^{2\pi}f(e^{i\theta})e^{-ik\theta}d\theta, \qquad i=\sqrt{-1}.$$
Then, in the terminology of Widom \cite{Widom}, the {\em determinant of a Toeplitz matrix with symbol $f$} is given by
$$E_n(f)=\det [u_{j-i}].$$
Toeplitz determinants find homes in many topics, such as statistical physics and random matrix theory. Different types of matrix ensembles lend themselves to classes of Toeplitz plus Hankel matrices of the form
$$\det [u_{j-i}\pm u_{j+i}], \qquad \det [u_{j-i}\pm u_{j+i+1}], \qquad \text{etc.}$$
From a more mathematical vantage point, such determinants arise in the decomposition of the determinants $E_n(f)$. Namely, if the symbol $f$ is an even function (in the sense $f(z)=f(\frac1z)$ on the unit circle), then
\begin{align*}
E_{2n+1}(f)&=\frac12\det [u_{j-i}-u_{j+i+2}]_{0}^{n-1} \times \det [u_{j-i}+u_{j+i}]_0^n, \\
E_{2n}(f)&=\,\,\,\,\, \det [u_{j-i}-u_{j+i+1}]_{0}^{n-1} \times \det [u_{j-i}+u_{j+i}]_0^{n-1}.
\end{align*}
Andrews and Stanton (\cite{Andrews}, see p. 274) involved a somewhat analogous splitting up of a certain determinant that arose in plane partitions although they did not view them as such.

\section{BONUS: CT Identities and Plane Partitions}

\noindent
Borrowing notations from Section 6, the Toeplitz matrix $[u_{i-j}]$ that appeared as a component in our matrix $T_{c,a}(b)$ is in fact recognizable due to its relevance in the theory of plane partitions. 
Indeed, MacMahon \cite{MacMahon} gave an elegant explicit  formula for the number of plane partitions inside an $a\times b\times c$ box:
$$\prod_{i=0}^{a-1}\prod_{j=0}^{b-1}\prod_{k=0}^{c-1}\frac{i+j+k+2}{i+j+k+1}.$$
One may readily associate a determinantal formulation \cite{Kratt4} to this enumeration
$$\det [u_{i-j}]=\det\left[\binom{a+b}{a-i+j}\right]_{i,j=0}^{c-1}=\prod_{i=0}^{a-1}\prod_{j=0}^{b-1}\prod_{k=0}^{c-1}\frac{i+j+k+2}{i+j+k+1}.$$
We have found what appears to be new (for us, at least) matrix whose determinant matches the above product formula and enjoys the expected symmetry.

\begin{lemma} \label{symmetry} It holds true that
$$\det\left[\binom{i+j+a+b}{i+a}\right]_{i,j=0}^{c-1}=\prod_{i=0}^{a-1}\prod_{j=0}^{b-1}\prod_{k=0}^{c-1}\frac{i+j+k+2}{i+j+k+1}.$$
\end{lemma}
\begin{proof} This is an immediate application of Dodgson's condensation \cite{Zeil1}.
\end{proof}

\begin{remark}
The third author has previously employed the same approach towards furnishing a simple proof \cite{Zeil2} of a determinant by MacMahon \cite{MacMahon}.
\end{remark}

\begin{proof} Here is a combinatorial proof for Lemma \ref{symmetry} given by C. Krattenthaler \cite{Kratt3}. We thank him for allowing us to reproduce it. 

\smallskip
\noindent
Plane partitions in an $a\times b\times c$ box are in bijection with
families $(P_0,P_1,\dots,P_{c-1})$ of non-intersecting lattice paths,
where $P_i$ runs from $(-i-a,i)$ to $(-i,i+b)$, $i=0,1,\dots,c-1$.
This is explained in Section~3.3 of Bressoud's book \cite{Bressoud}.
An example for such a family of paths for $a=6$ and $b=c=3$ is displayed in the
figure below.
$$
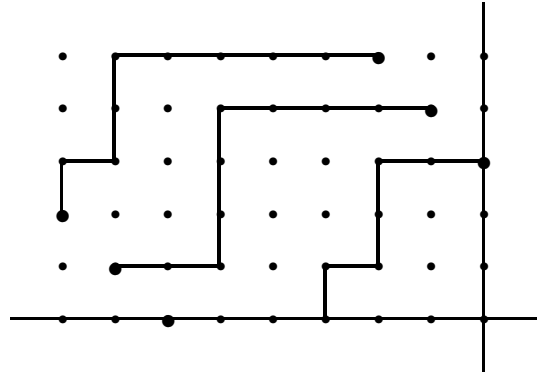
\Einheit=.7cm
\Gitter(1,6)(-8,0)
\Koordinatenachsen(1,6)(-8,0)
\Pfad(0,3),556656555\endPfad
\Pfad(-1,4),555566655\endPfad
\Pfad(-2,5),555556656\endPfad
\DickPunkt(-6,0)
\DickPunkt(-7,1)
\DickPunkt(-8,2)
\DickPunkt(0,3)
\DickPunkt(-1,4)
\DickPunkt(-2,5)
\hskip-5cm
$$
By the Lindstr\"om--Gessel--Viennot theorem \cite{Gessel}, this leads to the determinant
$$
\det_{0\le i,j\le c-1}\left(\binom {a+b} {a-i+j}\right)
$$
for the enumeration of both objects being counted here.

\noindent
We may prepend and append ``forced" path pieces, see the thick path portions in
the figure below.
$$
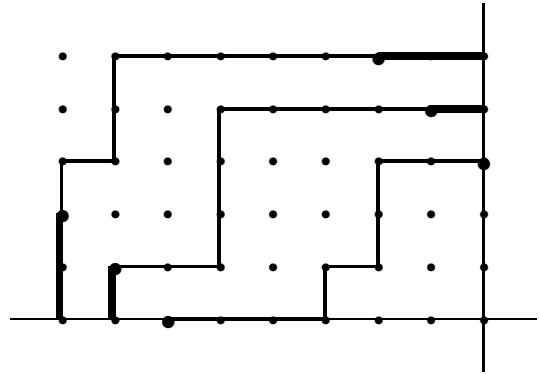
\Einheit=.7cm
\Gitter(1,6)(-8,0)
\Koordinatenachsen(1,6)(-8,0)
\Pfad(0,3),556656555\endPfad
\Pfad(-1,4),555566655\endPfad
\Pfad(-2,5),555556656\endPfad
\DickPunkt(-6,0)
\DickPunkt(-7,1)
\DickPunkt(-8,2)
\DickPunkt(0,3)
\DickPunkt(-1,4)
\DickPunkt(-2,5)
\PfadDicke{2.5pt}
\Pfad(-7,0),2\endPfad
\Pfad(-8,0),22\endPfad
\Pfad(0,4),5\endPfad
\Pfad(0,5),55\endPfad
\hskip-5cm
$$
In this manner, we obtain 
families $(Q_0,Q_1,\dots,Q_{c-1})$ of non-intersecting lattice paths,
where $Q_i$ runs from $(-i-a,0)$ to $(0,i+b)$, $i=0,1,\dots,c-1$.
The corresponding Lindstr\"om--Gessel--Viennot determinant \cite{Gessel} is
$$
\det_{0\le i,j\le c-1}\left(\binom {i+j+a+b} {i+a}\right).
$$
The proof is complete.
\end{proof}

\begin{remark} It is not hard to develop a $q$-analog of Lemma \ref{symmetry}:
$$\det\left[q^{i^2-ij}\binom{i+j+a+b}{i+a}_q\right]_{i,j=0}^{c-1}=\prod_{i=0}^{a-1}\prod_{j=0}^{b-1}\prod_{k=0}^{c-1}\frac{1-q^{i+j+k+2}}
{1-q^{i+j+k+1}}.$$
It also retains s similar proof as in above.

\end{remark}

\smallskip
\noindent
In the context of Selberg's integrals, there are equivalently-stated constant term identities, and we prove (a special case of) one due to Morris~\cite{Morris} associated with the root system $A_n$.

\begin{proposition} We have
\begin{align*}
CT_{\vec{t}}\,\prod_{i=0}^{c-1}(1+t_i)^a\left(1+t_i^{-1}\right)^b\prod_{ i\neq j}^{0,c-1}\left(1-t_jt_i^{-1}\right)^m 
=\prod_{\ell=0}^{c-1}\frac{(a+b+\ell m)!\,((\ell+1)m)!}{(a+\ell m)!\,(b+\ell m)!\, m!}.
\end{align*}
\end{proposition}
\begin{proof} We are only proving the special case $m=1$. Start with the determinant from Lemma \ref{symmetry} where the method of constant term extraction is brought to bear:
\begin{align*}
\det\left[\binom{i+j+a+b}{i+a}\right]&=\det\left[CT_{t_i} \left(t_i^{i+a}\left(1+\frac1{t_i}\right)^{i+j+a+b}\right)\right] \\
&=\det\left[CT_{t_i}\left((1+t_i)^{i+a}\left(1+\frac1{t_i}\right)^b\left(1+\frac1{t_i}\right)^j\right)\right] \\
&=CT_{\vec{t}}\,\,\,\prod_{i=0}^{c-1}(1+t_i)^{i+a}\left(1+\frac1{t_i}\right)^b\cdot \det\left[\left(1+\frac1{t_i}\right)^j\right] \\
&=CT_{\vec{t}}\,\,\,\prod_{i=0}^{c-1}(1+t_i)^{i+a}\left(1+\frac1{t_i}\right)^b\cdot \prod_{i<j}^{0,c-1}(t_j^{-1}-t_i^{-1}) \\
&=\frac1{c!}\,CT_{\vec{t}}\,\,\,\prod_{i=0}^{c-1}(1+t_i)^a\left(1+\frac1{t_i}\right)^b\cdot \prod_{i<j}^{0,c-1}(t_j-t_i)(t_j^{-1}-t_i^{-1}) \\
&=\frac1{c!}\,CT_{\vec{t}}\,\,\,\prod_{i=0}^{c-1}(1+t_i)^a\left(1+\frac1{t_i}\right)^b\cdot \prod_{0\leq i\neq j\leq c-1}\left(1-\frac{t_j}{t_i}\right).
\end{align*}
From Lemma \ref{symmetry}, again, we gather that
\begin{align*}
CT_{\vec{t}}\,\prod_{i=0}^{c-1}(1+t_i)^a\left(1+\frac1{t_i}\right)^b\cdot \prod_{i\neq j}^{0,c-1}\left(1-\frac{t_j}{t_i}\right)
=c!\cdot \prod_{i=0}^{a-1}\prod_{j=0}^{b-1}\prod_{k=0}^{c-1}\frac{i+j+k+2}{i+j+k+1}.
\end{align*}
Further algebraic manipulation confirms the desired conclusion.
\end{proof}

\noindent
We will strengthen the discussion by providing yet another new proof for a special case of a constant term identity due to Macdonald for the $BC_n$ root system~\cite{Morris}. We need some preliminary work first. To this end, recall the {\em super Catalan numbers}~\cite{Gessel2} defined by
$$S_{i,j}:=\frac{(2i)!\,(2j)!}{2\,i!\,j!\,(i+j)!}.$$
\begin{lemma} \label{pre-Mac} We have
$$\det [S_{i+a,j+b}]_{i,j=1}^n=\frac{(-1)^{\binom{n}2}}{2^n\,n!} \prod_{i=1}^n\frac{(2a+2i)!\,(2b+2i)!\, i!}{(a+i)!\,(b+i)!\,(a+b+n+i)!}.$$
\end{lemma}
\begin{proof} A direct application of Dodgson's condensation settles the argument. 
\end{proof}

\begin{corollary} Let $\vec{t}=(t_1,\dots,t_n)$. The following identity holds true
\begin{align*}
CT_{\vec{t}}\,&\prod_{i=1}^n(1-t_i)^{\alpha}\left(1-\frac1{t_i}\right)^{\alpha}(1-t_i^2)^{\beta}\left(1-\frac1{t_i^2}\right)^{\beta} \,\, \times \\
& \qquad \qquad \qquad \qquad \prod_{i<j}\left(1-\frac{t_j}{t_i}\right)\left(1-\frac{t_i}{t_j}\right)(1-t_it_j)\left(1-\frac1{t_it_j}\right) \\
=& \prod_{i=1}^n\frac{(2\beta+2i-2)!\,(2\alpha+2\beta+2i-2)!\,i!}{(\beta+i-1)!\,(\alpha+\beta+i-1)!\,(\alpha+2\beta+n+i-2)!}.
\end{align*}
\end{corollary}
\begin{proof} Start with the observation that
$$S_{i+a,j+b}=(-1)^{i+a}\,CT_t\frac{(1-t)^{2i+2a}(1+t)^{2j+2b}}{2\,t^{i+a+j+b}}$$
and proceed as usual
\begin{align*}
&\det [S_{i+a,j+b}]_1^n=\det\left[(-1)^{i+a}\,CT_{t_i}\frac{(1-t_i)^{2i+2a}(1+t_i)^{2j+2b}}{2\,t_i^{i+a+j+b}}\right]_1^n \\
&=\gamma_n\,\det\left[CT_{t_i}\frac{(1-t_i)^{2i+2a}(1+t_i)^{2b+2}}{t_i^{i+a+b+1}} \left(2+t_i+\frac1{t_i}\right)^{j-1}\right] \\
&=\gamma_n\, CT_{\vec{t}}\,\prod_{i=1}^n \frac{(1-t_i)^{2i+2a}(1+t_i)^{2b+2}}{t_i^{i+a+b+1}}\prod_{i<j}\left(t_j+\frac1{t_j}-t_i-\frac1{t_i}\right) \\
&=\gamma_n\, CT_{\vec{t}}\,\prod_{i=1}^n\frac{(1-t_i)^{2a+2}(1+t_i)^{2b+2}(-2+t_i+\frac1{t_i})^{i-1}}{t_i^{a+b+2}}\prod_{i<j}(t_j-t_i)\left(1-\frac1{t_it_j}\right) \\
&=\frac{\gamma_n}{n!}\, CT_{\vec{t}}\,\prod_{i=1}^n\frac{(1-t_i)^{2a+2}(1+t_i)^{2b+2}}{t_i^{a+b+2}}\prod_{i<j}(t_j-t_i)^2\left(1-\frac1{t_it_j}\right)^2 \\
&=\frac{(-1)^{n(a+1)}\gamma_n}{n!}\, CT_{\vec{t}}\,\prod_{i=1}^n(1-t_i)^{a+1}\left(1-\frac1{t_i}\right)^{a+1}(1+t_i)^{b+1}\left(1+\frac1{t_i}\right)^{b+1} \\
& \qquad \qquad \qquad \qquad \times \prod_{i<j}\left(1-\frac{t_j}{t_i}\right)\left(1-\frac{t_i}{t_j}\right)(1-t_it_j)\left(1-\frac1{t_it_j}\right);
\end{align*}
where $\gamma_n:=(-1)^{\binom{n+1}2}(-1)^{na}2^{-n}$. The rest follows from applying Lemma \ref{pre-Mac}, from above, with $a=\alpha+\beta-1$ and $b=\beta-1$ and some algebraic simplifications.
\end{proof}

\section{{\bf Postscript:} The Fourth Method - Ask Determinant guru Christian Krattenthaler}

\noindent
After the first version of this paper was written, we sent it out to a few experts, and to our initial {\it horror}, we got the following email from
Christian Krattenthaler, that he kindly allowed us to quote.

\smallskip
\noindent

{\it
``Obviously, I agree that determinants are {\rm (}*not* non-intuitive{\rm)} = intuitive, 

extremely useful, and that linear algebra without determinants is an aberration. 

{\rm (}We had in fact a colleague at our determinant {\rm(}sic!{\rm)} - working in harmonic 

analysis - who also believed in
- and taught - a determinant-free linear algebra; 

his poor students had to learn determinants afterwards ...{\rm)}

Concerning the determinant that you look at: it is the special case of {\rm(3.18)} in 

my ``Advanced Determinant Calculus'' \cite{Kratt2}
where $q=1$, $A=x+m$, $L_i=i-m$. 

This identity {\bf has} a simple direct proof. As explained above Theorem $30$ in \cite{Kratt2},

after having taken out appropriate factors, the determinant turns out to be a 

special case of Lemma 5 in \cite{Kratt2}. That lemma is so general so that it can be 

proved in various ways, one of which is explained in Ref. {\rm [88]} of \cite{Kratt2}.'' 
}

\smallskip
\noindent
Indeed this is an extreme case of {\it P\'olya's dictum} already alluded to above. The general lemma, in {\it hindsight}, is almost trivial to prove, {\it once known}, but coming up with it, is {\it anything but}.

\smallskip
\noindent
At first,  we {\it kicked ourselves} for not asking guru Christian right away, and saving us the {\it trouble} of doing it ourselves. But this is the cowardly way. This `trouble' was lots of fun, and
it brings us to yet another take-home message.

\smallskip
\noindent
{\bf Don't rush to ask the expert!}, it is like `peeking at the answer at the back of the textbook'.
You will learn much more by trying to do it {\it all by yourself!} Of course, if you finally give up, after an honest effort, then you can contact Professor Krattenthaler.
Also if you succeed, and want to make sure that you were not scooped, do likewise, but {\bf not right away!}.

\begin{remark}
The determinant (3.18) in \cite{Kratt2}, generalizing our Conjecture \ref{conj1}, was used in the proof of {\em refinements of the MacMahon ex-Conjecture} and the
{\em Bender-Knuth ex-Conjecture}. 
\end{remark}

\section{The Revenge of Rev. Charles}

\noindent
At times life rides like a domino-effect. The last section exhibited a notion of thinking retroactively - by way of consulting experts. This prompted us to have another look at the original problem
which would {\em really} make our paper even more self-contained.
In this section, we revisit our running example - the matrix
$$T_{n,m}(x):=\left[\binom{x+m}{j-i+m}-\binom{x+m}{m-i-j+1}\right]_{1\leq i,j\leq n},$$
for which we had the formula
$$\det T_{n,m}(x)=\prod_{i=1}^n\prod_{j=1}^m\frac{(x+i-j)(x+2i+j-2)}{(x+2i-j)(i+j-1)}.$$
We have had a ton of adventures in proving and applying it. As noted above, Krattenthaler claimed priority by providing a generalized determinant evaluation. Fair enough.
We, too, have another {\em idea} in trivializing the proof. Follow us.

\smallskip
\noindent
{\bf Step 1.} Change variables: $x=a+b, m=b-a$. Rename $T_{n,m}(x)$ to write
$$U_n(a,b)=\left[\binom{2b}{j-i+b-a}-\binom{2b}{b-a-i-j+1}\right]_{1\leq i,j\leq n}.$$
{\bf Step 2.} This is the {\color{red} key}: use the symmetry $\binom{A}{B}=\binom{A}{A-B}$ on the second binomial coefficient so that
$$U_n(a,b)=\left[\binom{2b}{j-i+b-a}-\binom{2b}{a+b+i+j-1}\right]_{1\leq i,j\leq n},$$
while maintaining the lingering assertion
$$\det U_n(a,b)=\prod_{i=1}^n\prod_{j=1}^{b-a}\frac{(a+b+i-j)(a+b+2i+j-2)}{(a+b+2i-j)(i+j-1)}.$$
{\bf Step 3.} Call Dodgson's condensation on $U_n(a,b)$ because this matrix is in a natural form for such an application. 
\begin{proof} Denoting the determinant by $M_n(a,b)$, Dodgson's condensation states that
$$M_n(a,b)=\frac{M_{n-1}(a,b)M_{n-1}(a+1,b+1)-M_{n-1}(a+1,b)M_{n-1}(a,b+1)}{M_{n-2}(a+1,b+1)}.$$
Therefore, we only need to verify (routinely) that the right-hand side does the same. To conclude, simply check both sides agree, say, when $n=1$ and $n=2$.
\end{proof}

\end{document}

\noindent
Well, why stop here: catch the $q$-disease. Recall $[n]_q!=\prod_{i=1}^n\frac{1-q^i}{1-q}$ and the {\em Gaussian polynomials} (or $q$-binomial coefficients) given by $\binom{n}k_q=\frac{[n]_q!}{[k]_q!\,[n-k]_q!}$. We are geared up to generalizing $T_{n,m}(x)$, by introducing three more parameters $a, b$ and $q$, in accordance to
$$\mathbf{T}_{n,m,a,b}(x;q):=\left[q^{u_{i,j}}\binom{x+m}{j+b-i-a+m}_q-q^{v_{i,j}}\binom{x+m}{m-i-a-j-b+1}_q\right]_{i,j=1}^n$$
where $u_{i,j}=(j+b)(j-i+b-a)$ and $v_{i,j}=(j+b)(j+i+b+a+x-m-1)$.

\smallskip
\noindent
{\bf Convention.} Define $\binom{t}s_q=0$ when $t<s$ or $t<0$; also $\prod_{i=s}^t(\cdot)=0$ while $t<s$.

\begin{proposition} Let $\varphi(n,a,b)=\frac12(b-a)(2b+n+1)n$ and $x$ be an indeterminate. If $m\leq a+b$ and $b\geq a-1$ then we have
\begin{align*}
\det \mathbf{T}_{n,m,a,b}(x;q)
=q^{\varphi(n,a,b)}\prod_{i=1}^n\prod_{j=1}^{m+b-a}\frac{(1-q^{x+m+i-j})(1-q^{x+2i+j+a-b-2})}{(1-q^{x+m+2i-j-1})(1-q^{i+j-1})}.
\end{align*}
If $a=b=0$ then it holds that
\begin{align*}
\det \mathbf{T}_{n,m,0,0}(x;q)
=\prod_{i=1}^n\prod_{j=1}^m\frac{(1-q^{x+i-j})(1-q^{x+2i+j+a-b-2})}{(1-q^{x+2i-j})(1-q^{i+j-1})}.
\end{align*}
\end{proposition}
\begin{proof} These are yet - other - application of Dodgson's condensation.
\end{proof}